\documentclass[11pt,a4paper]{article}
\usepackage[lmargin=1.0in,rmargin=1.0in,bottom=1.0in,top=1.0in,twoside=False]{geometry}
\usepackage[english]{babel}
\usepackage[shortlabels]{enumitem}
\usepackage[utf8]{inputenc}
\usepackage[T1]{fontenc}

\usepackage{fontaxes}
\usepackage{trimspaces}
\usepackage{nccfoots}
\usepackage{setspace}
\usepackage{inconsolata}
\usepackage{libertine}
\usepackage{amsmath,amssymb,amsthm}
\usepackage{mathtools}
\usepackage{comment}
\usepackage[all,defaultlines=4]{nowidow}
\usepackage{microtype}
\usepackage{mathrsfs}
\usepackage{mathtools}
\usepackage{marvosym}
\usepackage{thm-restate}
\usepackage{tikz}
\usepackage{todonotes}
\usetikzlibrary{calc}
\usetikzlibrary{shapes}
\definecolor{darkgreen}{rgb}{10,117,28}
\usepackage{wrapfig}
\usepackage{diagbox}
\usepackage{stackengine}
\usepackage{mathrsfs}

\definecolor{blue}{rgb}{0.1,0.2,0.5}
\definecolor{brown}{rgb}{0.6,0.6,0.2}
\usepackage[ocgcolorlinks, linkcolor={blue}, citecolor={blue}]{hyperref}
\usepackage{cleveref}
\usepackage{subcaption}
\crefformat{page}{#2page~#1#3}%
\Crefformat{page}{#2Page~#1#3}%
\crefformat{equation}{#2(#1)#3}%
\Crefformat{equation}{#2(#1)#3}%
\crefformat{figure}{#2Figure~#1#3}%
\Crefformat{figure}{#2Figure~#1#3}%
\crefformat{section}{#2Section~#1#3}
\Crefformat{section}{#2Section~#1#3}
\crefformat{chapter}{#2Chapter~#1#3}
\Crefformat{chapter}{#2Chapter~#1#3}
\crefformat{chapter*}{#2Chapter~#1#3}
\Crefformat{chapter*}{#2Chapter~#1#3}
\crefformat{part}{#2Part~#1#3}
\Crefformat{part}{#2Part~#1#3}
\crefformat{enumi}{#2(#1)#3}
\Crefformat{enumi}{#2(#1)#3}

\newtheorem{theorem}{Theorem}[section]
\crefformat{theorem}{#2Theorem~#1#3}
\Crefformat{theorem}{#2Theorem~#1#3}
\newtheorem{lemma}[theorem]{Lemma}
\crefformat{lemma}{#2Lemma~#1#3}
\Crefformat{lemma}{#2Lemma~#1#3}

\crefformat{conjecture}{#2Conjecture~#1#3}
\Crefformat{conjecture}{#2Conjecture~#1#3}

\newcommand{\newtheoremwithcrefformat}[2]{%
  \newtheorem{#1}{#2}[section]%
  \crefformat{#1}{##2\MakeUppercase#1~##1##3}%
  \Crefformat{#1}{##2\MakeUppercase#1~##1##3}%
}

\newtheoremwithcrefformat{proposition}{Proposition}
\newtheoremwithcrefformat{observation}{Observation}
\newtheoremwithcrefformat{corollary}{Corollary}
\newtheoremwithcrefformat{claim}{Claim}
\newtheoremwithcrefformat{example}{Example}
\newtheoremwithcrefformat{remark}{Remark}
\newtheoremwithcrefformat{definition}{Definition}

\makeatletter
\def\ifenv#1{
   \def\@tempa{#1}%
   \ifx\@tempa\@currenvir
      \expandafter\@firstoftwo
    \else
      \expandafter\@secondoftwo
   \fi
}
\let\wfs@comment@comment\comment
\let\comment\@undefined
\newcommand{\untoto}{\let\toto\@undefined}
\usepackage{changes}
\let\wfs@changes@comment\comment
\let\comment\@undefined

\newcommand\comment{%
    \ifthenelse{\equal{\@currenvir}{comment}}
    {\wfs@comment@comment}
    {\wfs@changes@comment}%
}
\makeatother

\usepackage{anyfontsize}

\newcommand{\Cc}{\mathscr{C}}

\renewcommand{\phi}{\varphi}

\renewcommand{\leq}{\leqslant}
\renewcommand{\geq}{\geqslant}

\newcommand{\tup}[1]{\bar{#1}}

\newcommand{\Z}{\mathbb{Z}}
\newcommand{\R}{\mathbb{R}}
\newcommand{\N}{\mathbb{N}}

\usepackage{stmaryrd}

\newcommand{\sca}[2]{\langle #1,#2\rangle}

\renewcommand{\setminus}{-}

\newcommand{\Oh}{\mathcal{O}}

\newcommand{\Pp}{\mathcal{P}}
\newcommand{\Qq}{\mathcal{Q}}
\newcommand{\Ff}{\mathcal{F}}

\newcommand{\WReach}{\mathrm{WReach}}
\newcommand{\SReach}{\mathrm{SReach}}
\newcommand{\adm}{\mathrm{adm}}
\newcommand{\scol}{\mathrm{scol}}
\newcommand{\wcol}{\mathrm{wcol}}
\newcommand{\pleq}{\preceq}

\usepackage{accents}

\renewcommand{\preceq}{\preccurlyeq}

\newcommand{\citeleq}[1]{\stackrel{\mathclap{\normalfont\mbox{\cite{#1}}}}{\leq}}

\newcommand{\zz}{\mathbf{0}}

\newcommand{\ERCagreement}{
\vspace{-6pt}

\noindent
{\begin{minipage}[t]{\textwidth}\small This paper is a part of projects {\sc{CRACKNP}} (JN), {\sc{BOBR}} (MP), and {\sc{TIPEA}} (KW) that have received funding from the European Research Council (ERC) under the European Union's Horizon 2020 research and innovation programme (grant agreements No~853234,~948057, and~850979 respectively). \end{minipage}\hfill

\begin{minipage}{.32\textwidth}\includegraphics[width=\textwidth]{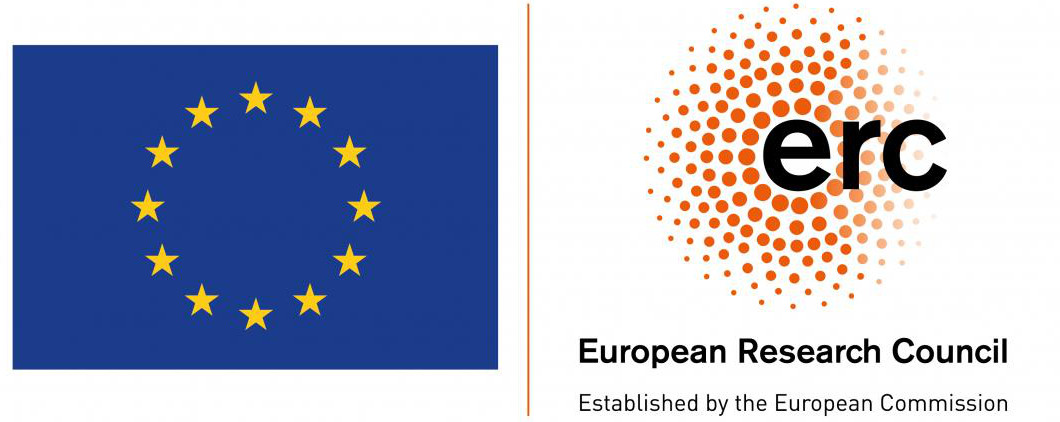}\end{minipage}\hfill}}

\makeatletter
\newcommand{\@abbrev}[3]{
  \def\c@a@def##1{
      \if ##1.
        \relax
      \else
        \@ifdefinable{\@nameuse{#1##1}}{\@namedef{#1##1}{#2##1}}
        \expandafter\c@a@def
      \fi
    }
  \c@a@def #3.
}
\@abbrev{bb}{\mathbb}{ABCDEFGHIJKLMNOPQRSTUVWXYZ}
\@abbrev{bf}{\mathbf}{ABCDEFGHIJKLMNOPQRSTUVWXYZabcdefghijklmnopqrstuvwxyz}
\@abbrev{bit}{\boldsymbol}{ABCDEFGHIJKLMNOPQRSTUVWXYZabcdefghijklmnopqrstuvwxyz}
\@abbrev{cal}{\mathcal}{ABCDEFGHIJKLMNOPQRSTUVWXYZ}
\@abbrev{frak}{\mathfrak}{ABCDEFGHIJKLMNOPQRSTUVWXYZabcdefghijklmnopqrstuvwxyz}
\@abbrev{rm}{\mathrm}{ABCDEFGHIJKLMNOPQRSTUVWXYZabcdefghijklmnopqrstuvwxyz}
\@abbrev{scr}{\mathscr}{ABCDEFGHIJKLMNOPQRSTUVWXYZ}
\@abbrev{sf}{\mathsf}{ABCDEFGHIJKLMNOPQRSTUVWXYZabcdefghijklmnopqrstuvwxyz}
\@abbrev{Alg}{\Alg}{ABCDEFGHIJKLMNOPQRSTUVWXYZ}
\@abbrev{Str}{\Str}{ABCDEFGHIJKLMNOPQRSTUVWXYZ}
\@abbrev{set}{\mathbb}{ABCDEFGHIJKLMNOPQRSTUVWXYZ}
\@abbrev{tup}{\tup}{ABCDEFGHIJKLMNOPQRSTUVWXYZabcdefghijklmnopqrstuvwxyz}
\@abbrev{tld}{\tilde}{ABCDEFGHIJKLMNOPQRSTUVWXYZabcdefghijklmnopqrstuvwxyz}
\makeatother

\usepackage{wasysym}
\usepackage{xspace}

\title{Bounding generalized coloring numbers of planar graphs\\ using coin models}

\author{
Jesper Nederlof\thanks{Utrecht University, The Netherlands, \textit{j.nederlof@uu.nl}}
\and
Micha\l{} Pilipczuk\thanks{University of Warsaw, Poland, \textit{michal.pilipczuk@mimuw.edu.pl}}
\and
Karol Węgrzycki\thanks{Saarland University and Max Planck Institute for
Informatics,
Saarbr\"ucken, Germany, \textit{wegrzycki@cs.uni-saarland.de}
}
}

\date{}

\begin{document}

\maketitle

\begin{abstract}
 We study {\em{Koebe orderings}} of planar graphs: vertex orderings obtained by modelling the graph as the intersection graph of pairwise internally-disjoint discs in the plane, and ordering the vertices by non-increasing radii of the associated discs. We prove that for every $d\in \N$, any such ordering has $d$-admissibility bounded by $\Oh(d/\ln d)$ and weak $d$-coloring number bounded by $\Oh(d^4 \ln d)$. This in particular shows that the $d$-admissibility of planar graphs is bounded by $\Oh(d/\ln d)$, which asymptotically matches a known lower bound due to Dvo\v{r}\'ak and Siebertz.  
\end{abstract}

\vfill
\ERCagreement

\pagebreak

\section{Introduction}

The {\em{degeneracy}} of a vertex ordering of a graph $G$ is the maximum number of neighbors that any vertex $v$ has among vertices smaller than $v$ in the ordering. The degeneracy of $G$ is the minimum possible degeneracy of a vertex ordering of $G$. If one takes a vertex ordering of $G$, say of degeneracy $k$, and applies a greedy left-to-right coloring procedure, then the obtained proper coloring of $G$ uses at most $k+1$ distinct colors. For this reason, the {\em{coloring number}} --- defined as degeneracy plus $1$ --- is an upper bound on the chromatic number of a graph.

In~\cite{KiersteadY03}, Kierstead and Yang introduced {\em{generalized coloring numbers}}, which extend the concept of degeneracy/coloring number by replacing measuring the number of smaller (with respect to the fixed ordering) neighbors by measuring the number of smaller vertices reachable by short paths. Here, reachability can be understood in various ways, but the following notions provide a robust set of definitions. If $G$ is a graph and $\pleq$ is a vertex ordering of $G$, then a {\em{strong reachability path}} from a vertex $v$ to some vertex $u\pleq v$ is a path $P$ in $G$ that starts at $v$, finishes at $u$, and such that all vertices of $P$ except for $u$ are not smaller than $v$ in $\pleq$. A~{\em{weak reachability path}} is defined similarly, except that we only require that all vertices of $P$ are not smaller than $u$ in $\pleq$. Thus, we allow $P$ to use vertices between $u$ and $v$ in $\pleq$. We say that $u$ is {\em{strongly $d$-reachable}} from $v$ in $\pleq$ if there is a strong reachability path of length at most $d$ from $v$ to $u$; weak $d$-reachability is defined analogously. Finally, the strong (resp. weak) $d$-coloring number of $\pleq$ is the maximum number of strongly (resp. weakly) $d$-reachable vertices from any vertex of $G$, and the strong (resp. weak) $d$-coloring number of $G$ is the minimum possible strong (resp. weak) $d$-coloring number of a vertex ordering of $G$. In this work we will be also interested in the concept of {\em{$d$-admissibility}}, introduced later by Dvo\v{r}\'ak~\cite{Dvorak13}, which is defined by measuring the largest possible size of a family of strong reachability paths of length at most $d$ (i.e., paths having at most $d$ edges) that share the origin, but otherwise are pairwise disjoint.

It turns out that all three parameters defined above --- strong $d$-coloring number, weak $d$-coloring number, and $d$-admissibility --- are functionally equivalent for every fixed $d$ (see e.g.~\cite{notes}). Moreover, they can be used to characterize classes of {\em{bounded expansion}}: the concept of uniform sparseness in graphs that is central to the theory of sparse graph classes of Ne\v{s}et\v{r}il and Ossona de Mendez. More precisely, as observed by Yang~\cite{Yang09b} and by Zhu~\cite{Zhu09}, a class of graphs $\Cc$ has bounded expansion if and only if for every fixed $d\in \N$ there is a uniform upper bound on the weak $d$-coloring number (equivalently, on the strong $d$-coloring number or $d$-admissibility) of graphs in $\Cc$. For this reason, the generalized coloring number have become a key technical tool in the area of Sparsity, with multiple combinatorial and algorithmic applications; see e.g.~\cite{AmiriMRS18,BrianskiMPS21,Dreier21,Dvorak13,EickmeyerGKKPRS17,JoretMMW19,KwonPS20,NesetrilMPZ20,PilipczukST18,ReidlVS19}. We refer the reader to appropriate chapters of~\cite{sparsity} and of~\cite{notes} for an overview of basic properties and applications of generalized coloring numbers.

\newcommand{\Planar}{\mathrm{Planar}}

The generic arguments used in~\cite{Yang09b,Zhu09} to bound the generalized coloring numbers in bounded expansion classes provide only very crude upper bounds on their values. These upper bounds are typically far from optimal, which motivates the search for tighter asymptotic estimates  on various well-studied classes of sparse graphs. Among these, perhaps the most interesting case is that of planar graphs. And so, if by $\adm_d(\Planar)$, $\scol_d(\Planar)$, and $\wcol_d(\Planar)$ we respectively denote the maximum $d$-admissibility, strong $d$-coloring number, and weak $d$-coloring number among planar graphs, then
the following lower and upper bounds have been known so far:
\begin{align*}
 \Omega(d/\ln d)\citeleq{adm-private} \adm_d(\Planar) \citeleq{HeuvelMQRS17} \Oh(d);\ \\
 \Omega(d) \citeleq{HeuvelMQRS17} \scol_d(\Planar) \citeleq{HeuvelMQRS17} \Oh(d);\ \\
 \Omega(d^2\ln d) \citeleq{JoretM21} \wcol_d(\Planar) \citeleq{HeuvelMQRS17} \Oh(d^3).
\end{align*}
Thus, only the asymptotics of the strong $d$-coloring numbers have been determined precisely. We note that the lower bound on the strong $d$-coloring number is only sketched in~\cite{HeuvelMQRS17}, while the lower bound on the $d$-admissibility was communicated to us by Zden\v{e}k Dvo\v{r}\'ak and Sebastian Siebertz~\cite{adm-private} and has not been published. Therefore, for completeness, in Appendix~\ref{app:scol} we give proofs of both these results.

\paragraph*{Our contribution.} So far, the best upper bounds for generalized coloring numbers on planar graphs are provided by the work of van den Heuvel et al.~\cite{HeuvelMQRS17} and use purely graph-theoretic decomposition methods. In this work we turn to a more geometric approach by studying {\em{Koebe orderings}} of planar graphs. More precisely, the celebrated theorem of Koebe~\cite{Koebe30} states that every planar graph has a {\em{coin model}}: with every vertex $u$ one can associate a disc $D(u)$ in the plane so that those discs are pairwise internally disjoint, and whenever $u$ and $v$ are adjacent in the graph, the corresponding discs $D(u)$ and $D(v)$ are tangent. Given such a model, one can define a very natural vertex ordering: just order the vertices by non-increasing radii of the associated discs. (Equi-sized discs are ordered arbitrarily.) Every vertex ordering of a planar graph that can be constructed in this way shall be called a {\em{Koebe ordering}}.

Studying Koebe orderings in the context of sparse graphs is not entirely new. A natural source of examples of non-trivial sparse graphs comes from studying intersection graphs of families of geometric objects in Euclidean spaces. Here, usually one assumes that the family if {\em{$c$-thin}} for some constant $c$, that is, every point in the space is contained in at most $c$ objects. For instance, Koebe's theorem implies that every planar graph is isomorphic to a subgraph of the intersection graph of a $2$-thin family of discs in the plane. In~\cite{DvorakMN21}, Dvo\v{r}\'ak et al. studied separator properties of such geometric intersection~graphs. The work~\cite{DvorakMN21}, similarly to this one, was partially motivated by the beautiful proof of the Lipton-Tarjan Separator Theorem using Koebe's Theorem, due to Har-Peled~\cite{Har-Peled11}. More recently, Dvo\v{r}\'ak et al.~\cite{DvorakPUY21} studied generalized coloring numbers of geometric intersection graphs, and ordering objects from largest to smallest was a recurring idea. From this perspective, our work can be regarded as an application of this idea to the specific case of planar graphs and their coin models, in search for tighter~bounds.

And so, we prove that for every $d\in \N$ and every Koebe ordering $\pleq$ of a planar graph $G$, we have
\begin{enumerate}[label=(\arabic*),ref=(\arabic*),leftmargin=*]
 \item\label{r:adm}  $\adm_d(G,\pleq)\leq \Oh(d/\ln d)$;
 \item\label{r:scol} $\scol_d(G,\pleq)\leq \Oh(d^2)$; and
 \item\label{r:wcol} $\wcol_d(G,\pleq)\leq \Oh(d^4\ln d)$.
\end{enumerate}
Result~\ref{r:adm} is probably the most interesting contribution, as it improves the state-of-the-art upper bounds on the $d$-admissibility of planar graphs to tightness. Note that so far only an $\Oh(d)$ upper bound was known, which followed from bounding the strong $d$-coloring number, while an $\Omega(d/\ln d)$ lower bound was given recently by Dvo\v{r}\'ak and Siebertz~\cite{adm-private}. Thus, result~\ref{r:adm} asymptotically closes the gap between lower and upper bounds on $d$-admissibility of planar graphs, and shows that Koebe orderings are asymptotically optimal in this context. The proof relies on a very careful area argument, where the notion of area is redefined using an appropriate density function.
Result~\ref{r:scol} follows from a very simple area argument that has already been observed in~\cite{DvorakPUY21}, so this is not a new result. Finally, as for result~\ref{r:wcol}, a general statement proved in~\cite{DvorakPUY21} implies an upper bound of $\Oh(d^{8} \ln d)$. Our proof applies a more careful analysis that involves geometric arguments specific to coin models.

We also provide some simple lower bounds for results~\ref{r:scol} and~\ref{r:wcol}. More precisely, we show that there are planar graphs and their coin models such that every Koebe ordering constructed based on those coin models achieves strong $d$-coloring number $\Omega(d^2)$ and weak $d$-coloring number $\Omega(d^3)$. Thus, for planar graphs, Koebe orderings are provably {\em{not}} asymptotically optimal for the strong $d$-coloring number (for which an $\Oh(d)$ upper bound can be obtained using graph-theoretic methods~\cite{HeuvelMQRS17}), and cannot surpass the current upper bound of $\Oh(d^3)$ on the weak $d$-coloring number. However, we were unable to find a lower bound higher than cubic in $d$, which makes us believe that it is possible that every Koebe ordering of a planar graph has weak $d$-coloring number bounded by $d^3\cdot \ln^{\Oh(1)} d$. Since the approach via Koebe orderings is radically different from previous upper bound techniques~\cite{HeuvelMQRS17}, we hope that it is possible to build upon this approach to provide a subcubic upper bound on the weak $d$-coloring number of planar graphs.

\section{Preliminaries}

We use standard graph notation; see for example the textbook by Diestel~\cite{graphtheoryDiestel}.

\paragraph*{Generalized coloring numbers.}
For a graph $G$, a {\em{vertex ordering}} of $G$ is a total order on the vertex set of $G$. Suppose $\pleq$ is a vertex ordering of $G$ and $d$ is a positive integer. For a vertex $v$, a {\em{weak reachability path}} starting at $v$ is a path $P$ in $G$ that starts at $v$, ends in a vertex $u\pleq v$, and such that all vertices of $P$ are not smaller in $\pleq$ than~$u$. A {\em{strong reachability path}} starting at $v$ is defined in the same way, except that we require all internal vertices of $P$ to be larger in $\pleq$ than $v$. For a positive integer $d$, we say that $v$ {\em{weakly $d$-reaches}} $u$ if there is a weak reachability path of length at most $d$ that starts at $v$ and ends at $u$. Strong $d$-reachability is defined analogously using strong reachability paths.
We define the following objects and quantities:
\begin{itemize}[nosep]
 \item The {\em{weak reachability set}} of $v$, denoted $\WReach^{G,\pleq}_d[v]$, is the set of all vertices $u$ that are weakly $d$-reachable from $v$.
 \item The {\em{strong reachability set}} of $v$, denoted $\SReach^{G,\pleq}_d[v]$, is the set of all vertices $u$ that are strongly $d$-reachable from $v$.
 \item The {\em{$d$-admissibility}} of $v$, denoted $\adm^{G,\pleq}_d(v)$, is the maximum size of a family of strong reachability paths of length $d$ that start at $v$ and are vertex-disjoint apart from sharing $v$.
\end{itemize}
With the \emph{length} of a path, we refer to the number of edges in it.
We may omit the superscript if the graph and the vertex ordering is clear from the context. Finally, the {\em{weak $d$-coloring number}} of a vertex ordering $\pleq$ in a graph $G$ is defined as $\max_{v\in V(G)} |\WReach^{G,\pleq}_d[v]|$, and the {\em{weak $d$-coloring number}} of $G$ is the minimum weak $d$-coloring number of a vertex ordering of~$G$. These are denoted by $\wcol_d(G,\pleq)$ and $\wcol_d(G)$, respectively. The strong $d$-coloring number and the $d$-admissibility of (a vertex ordering of) a graph $G$ are defined and denoted analogously.

\paragraph*{Coin models and Koebe's theorem.} A {\em{coin model}} for a graph $G$ is a mapping $D(\cdot)$ that assigns to each vertex $u$ of $G$ a circle $D(u)\subseteq \R^2$ (called a \emph{disk}) so that the following properties are satisfied:
\begin{itemize}[nosep]
 \item the discs $\{D(u)\colon u\in V(G)\}$ have pairwise disjoint interiors; and
 \item if vertices $u$ and $v$ in $G$ are adjacent, then discs $D(u)$ and $D(v)$ intersect.\\ (We say that they {\em{touch}}.)
\end{itemize}
Note that if a graph $G$ has a coin model, then $G$ is necessarily planar. The classic result of Koebe, which is the main inspiration for this work, shows that the converse is also true.

\begin{theorem}[Koebe's Theorem,~\cite{Koebe30}]\label{thm:koebe}
 Every planar graph has a coin model.
\end{theorem}

For a planar graph $G$, let a {\em{Koebe ordering}} of $G$ be any vertex ordering $\pleq$ constructed as follows: take any coin model $D(\cdot)$ of $G$ and let $\pleq$ be any vertex ordering such that whenever $D(u)$ has a strictly larger radius than $D(v)$, we have $u\pleq v$. So we order the vertices by non-increasing radii of the associated discs, but discs with same radii can be ordered arbitrarily. 
Throughout the paper, we will often denote $u < v$ to refer to $u$ being earlier than $v$ in a fixed Koebe ordering that is clear from the context.

Koebe's Theorem allows us to approach combinatorial problems in planar graphs using the toolbox of Euclidean geometry on $\R^2$.
Whenever considering $\R^2$, we equip it with the standard scalar product $\sca{\cdot}{\cdot}$, $\ell_2$ norm $\|x\|=\sqrt{\sca{x}{x}}$, and the standard Lebesgue measure $\lambda$. The point $(0,0)$ is denoted $\zz$.

\section{Admissibility}\label{sec:adm}

In this section we prove the following theorem, to which the remainder of this section is devoted.

\begin{theorem}\label{thm:adm-ub}
 Let $d\in \N$, $G$ be a planar graph, and $\pleq$ be any Koebe ordering of $G$. Then $$\adm_d(G,\pleq)\leq \Oh(d/\ln d).$$
 In particular, for every planar graph $G$, $\adm_d(G)\leq \Oh(d/\ln d)$.
\end{theorem}

Let $D(\cdot)$ be the coin model using which the ordering $\pleq$ was constructed. If suffices to prove that for every vertex $u$, we have $\adm_d^{G,\pleq}(u)\leq \Oh(d/\ln d)$. Let us fix $u$ from now on. By scaling and translation we may assume that $D(u)$ is the disk with center $\zz$ and radius $1$. Let $\Ff$ be a family of paths witnessing the value of $d$-admissibility of $u$. That is, $\Ff$ consists of paths $P_1,P_2,\ldots,P_k$, pairwise disjoint apart from sharing $u$, such that each $P_j$ has length at most $d$, starts at $u$, ends at a vertex $v_j$ such that $v_j<u$, and all internal vertices of $P_j$ are larger than $u$ in $\pleq$. Our goal is to prove that $k\leq \Oh(d/\ln d)$.

Consider any $j\in \{1,\ldots,k\}$.
Since $v_j\pleq u$, $D(v_j)$ has radius at least $1$. Let $x_j$ be the unique point of intersection of $D(v_j)$ and the disk of the predecessor of $v_j$ on $P_j$, and let $D'_j$ be the disk of radius $1$ that is entirely contained in $D(v_j)$ and contains the point $x_j$. (That is, $D'_j$ is an image of $D(v_j)$ in a homothety centered at $x_j$ with positive scale chosen so that $D'_j$ has radius $1$.) Let $w_1,\ldots,w_\ell$ be the internal vertices of the path $P_j$, so $\ell \leq  d-1$. Let
$$A_j\coloneqq D'_j\cup \bigcup_{i=1}^\ell D(w_i).$$
Since $w_i>u$ for each $i\in \{1,\ldots,\ell\}$, $A_j$ is a union of a sequence of at most $d$ disjoint disks of radii at most $1$, where every two consecutive disks touch and the first disk touches $D(u)$. It follows that
$$A_j\subseteq R,$$ 
where 
$$R=\{x\in \R^2~|~1\leq \|x\|\leq 2d+1\}.$$
Moreover, observe that sets $A_j$ have pairwise disjoint interiors since the paths $P_j$ have disjoint sets of vertices.

 \begin{figure}[ht]
 \centering
  \includegraphics[width=0.61\textwidth]{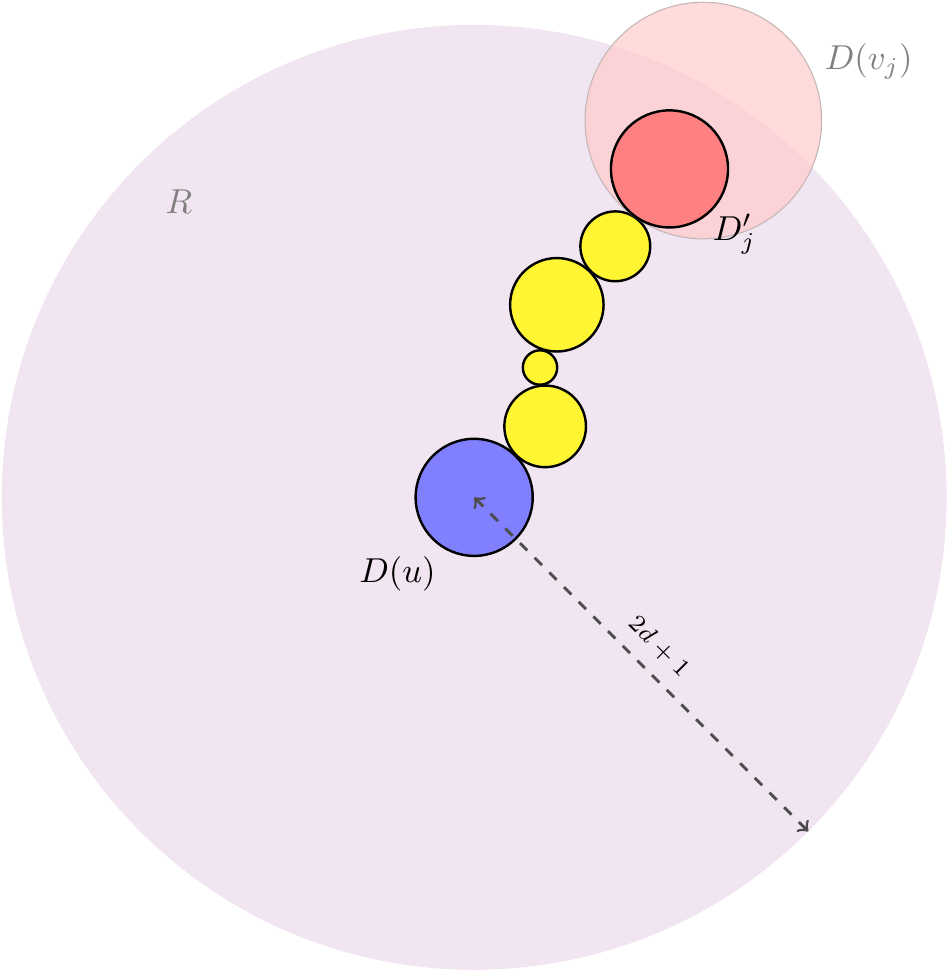}
  \caption{Example image of a path $P_j$ in the coin model. The disks of internal vertices --- $D(w_1),\ldots,D(w_{\ell})$ --- are depicted in yellow.}\label{fig:adm}
 \end{figure}

Let $g\colon \R^2\setminus \{\zz\}\to \R^+$ be the density defined as
$$g(x)=\frac{1}{\|x\|^2}.$$
We define a measure $\mu$ on $\R^2\setminus \{\zz\}$ as the measure with density $g$. That is, for a measurable set $L\subseteq \R^2\setminus \{\zz\}$, we set
$$\mu(L)=\int_L g\ d\lambda,$$
where $\lambda$ is the Lebesgue measure on $\R^2$.
Observe that
$$\mu(R)=\int_{1}^{2d+1} \frac{2\pi t}{t^2}\ dt=2\pi\ln (2d+1).$$
Therefore, to argue that $k\leq \Oh(d/\ln d)$ and thereby prove \cref{thm:adm-ub}, it suffices to show the following statement.

\begin{lemma}\label{lem:each-Ai}
 For every $j\in \{1,\ldots,k\}$, we have
 $$\mu(A_j)\geq \Omega\left(\frac{\ln^2 d}{d}\right).$$
\end{lemma}

From now on we focus on proving \cref{lem:each-Ai}.

\medskip

We will need the following auxiliary claim that provides a lower bound for the measure of disks.

\begin{claim}\label{cl:mu-lb}
 Let $D$ be a disk in $\R^2$ that has radius $\rho$, center at distance $a$ from $\zz$, and does not contain $\zz$. Then
 $$\mu(D)\geq \frac{\pi}{4}\cdot \frac{\rho^2}{a^2}.$$
\end{claim}
\begin{proof}
 Note that since $\zz\notin D$, we have $\rho<a$. It follows that each $x\in D$ is at distance at most $a+\rho<2a$ from $\zz$, implying that $g(x)>\frac{1}{4a^2}$. Therefore,
 $$\mu(D)\geq \frac{1}{4a^2}\cdot \lambda(D)=\frac{\pi}{4}\cdot \frac{\rho^2}{a^2}.\qedhere$$
\end{proof}

Recall that $A_j$ is the union of a sequence of disks $D(w_1),\ldots,D(w_\ell),D'_j$, where $\ell<d$. Denote them as $D_1,\ldots,D_{\ell+1}$ in order for convenience. Let $\rho_i$ be the radius of disk $D_i$ and $a_i$ be the distance between the center of $D_i$ and $\zz$. Note that since $D_1$ touches $D(u)$ and $D_i$ touches $D_{i-1}$ for $i=2,\ldots,\ell+1$, we have
$$a_i\leq 1+2\rho_1+2\rho_2+\ldots+2\rho_{i-1}+\rho_i\qquad \textrm{for all }i\in \{1,\ldots,\ell+1\}.$$ 
By \cref{cl:mu-lb}, we conclude that
\begin{eqnarray*}
\mu(A_j) & \geq & \frac{\pi}{4}\cdot \left(\frac{\rho_1^2}{(1+\rho_1)^2}+\frac{\rho_2^2}{(1+2\rho_1+\rho_2)^2}+\ldots+\frac{\rho_{\ell+1}^2}{(1+2\rho_1+2\rho_2+\ldots+2\rho_{\ell}+\rho_{\ell+1})^2}\right) \\
& \geq & \frac{\pi}{16}\cdot \left( \frac{\rho_1^2}{(1+\rho_1)^2}+\frac{\rho_2^2}{(1+\rho_1+\rho_2)^2}+\ldots+\frac{\rho_{\ell+1}^2}{(1+\rho_1+\rho_2+\ldots+\rho_{\ell}+\rho_{\ell+1})^2}\right).
\end{eqnarray*}
As $\rho_i\in [0,1]$ and $\rho_{\ell+1}=1$, to prove \cref{lem:each-Ai} it suffices to prove\footnote{The proof idea presented below was suggested to us by Karl Bringmann; we are grateful to Karl for this elegant argument that replaced our previous, more cumbersome reasoning.} the following purely analytic fact.

\begin{lemma}\label{lem:rhos}
 Let $\rho_1,\ldots,\rho_{\ell+1}\in [0,1]$ be such that $\rho_{\ell+1}=1$. Then 
 $$\frac{\rho_1^2}{(1+\rho_1)^2}+\frac{\rho_2^2}{(1+\rho_1+\rho_2)^2}+\ldots+\frac{\rho_{\ell+1}^2}{(1+\rho_1+\rho_2+\ldots+\rho_{\ell}+\rho_{\ell+1})^2}\geq \Omega\left(\frac{\ln^2 \ell}{\ell}\right).$$
\end{lemma}
\begin{proof}
 For $i\in \{0,1,\ldots,\ell+1\}$, let $x_i=\sum_{j=1}^i \rho_j$; thus $\rho_i=x_i-x_{i-1}$. Also, denote $q=1+x_{\ell+1}$. Observe that if $q\leq \ell^{1/3}$, then we have
 $$\sum_{i=1}^{\ell+1} \frac{\rho_i^2}{(1+x_i)^2}\geq \frac{\rho_{\ell+1}^2}{(1+x_{\ell+1})^2}=\frac{1}{q^2}\geq \frac{1}{\ell^{2/3}}\in \Omega\left(\frac{\ln^2 \ell}{\ell}\right).$$
 Hence, from now on we may assume that $q>\ell^{1/3}$.
 
 Observe that as $\rho_i\in [0,1]$ for each $i\in \{1,\ldots,\ell+1\}$, it holds that $1+x_i\leq 2(1+x_{i-1})$. Since function $f(t)=1/t$ is non-increasing, we have
 $$\sum_{i=1}^{\ell+1} \frac{\rho_i}{1+x_i}\geq \frac{1}{2} \sum_{i=1}^{\ell+1} \frac{\rho_i}{1+x_{i-1}}=\frac{1}{2} \sum_{i=1}^{\ell+1} (x_i-x_{i-1})\cdot f(1+x_{i-1})\geq \frac{1}{2} \int_{1}^q f(t)\ dt = \frac{1}{2}\ln q >\frac{1}{6} \ln \ell.$$
 Hence, we may use the AM-QM inequality to conclude that
 $$\sum_{i=1}^{\ell+1} \frac{\rho_i^2}{(1+x_i)^2}\geq \frac{1}{\ell+1}\cdot \left(\sum_{i=1}^{\ell+1} \frac{\rho_i}{1+x_i}\right)^2\geq \frac{\ln^2 \ell}{36(\ell+1)}\in \Omega\left(\frac{\ln^2 \ell}{\ell}\right).\qedhere$$
\end{proof}

As argued, \cref{lem:each-Ai} follows from \cref{lem:rhos}. So the proof of \cref{thm:adm-ub} is also complete.

\section{Strong and weak coloring numbers}

We start with a very simple upper bound for the strong coloring numbers. This result follows directly from a more general statement proved by Dvo\v{r}\'ak et al.~\cite[Lemma~1]{DvorakPUY21} using the same volume argument, so we include it here only for completeness.

\begin{theorem}\label{thm:scol-ub}
 Let $d\in \N$, $G$ be a planar graph, and $\pleq$ be any Koebe ordering of $G$. Then $$\scol_d(G,\pleq)\leq (2d+1)^2.$$
\end{theorem}
\begin{proof}
 Fix any vertex $u$ of $G$; our goal is to prove that $|\SReach_d[u]|\leq (2d+1)^2$. By scaling and translation, we may assume that $D(u)$ is the disc of radius $1$ centered at $\zz$. Consider any $v\in \SReach_d[u]$ and let $P$ be a strong reachability path witnessing this membership. Since $v\pleq u$, the radius of $D(v)$ is not smaller than that of $D(u)$, that is, it is at least $1$. Therefore, there is a disc $D'(v)$ of radius $1$ that is entirely contained in $D(v)$ and that touches the disc of the predecessor of $v$ on $P$.
 
 Observe that all vertices on $P$ apart from $v$ are not smaller in $\pleq$ than $u$, so their discs have radii at most $1$. By the triangle inequality it follows that the center of $D'(v)$ is at distance at most $2d$ from $\zz$, so in particular $D'(v)$ is entirely contained in the ball
 $$B\coloneqq \{\, x\in \R^2~|~\|x\|\leq 2d+1\}.$$
 Note that discs $\{D'(v)\colon v\in \SReach_d[v]\}$ have pairwise disjoint interiors and each of them has area (i.e. $\lambda$ measure) equal to $\pi$. It follows that
 $$|\SReach_d[u]|\leq \frac{\lambda(B)}{\pi}=(2d+1)^2.\qedhere$$
\end{proof}

As argued in~\cite{HeuvelMQRS17}, in fact every planar graph $G$ has a vertex ordering $\pleq$ satisfying $\scol_d(G,\pleq)\leq \Oh(d)$ for every $d$, and this bound is asymptotically tight. On the other hand, it is not hard to construct an example showing that the quadratic dependence on $d$ in \cref{thm:scol-ub} cannot be avoided if we restrict attention to Koebe orderings; see \cref{prop:scol-lb}. This shows that for the strong coloring number of planar graphs, Koebe orderings are not asymptotically optimal.

\medskip

We now turn attention to the weak coloring numbers. It follows from~\cite[Theorem~3]{DvorakPUY21} that for any fixed $d\in \N$ and any Koebe ordering $\pleq$ of a planar graph $G$, we have $\wcol_d(G,\pleq)\leq \Oh(d^{8}\ln d)$. However, the arguments used in~\cite{DvorakPUY21} apply to a more general setting of intersection graphs of thin families of convex objects in Euclidean spaces, so it is not surprising that in the concrete setting of coin models, a tighter upper bound can be obtained by a more careful geometric analysis. This we show in the next statement.

\begin{theorem}\label{thm:wcol-ub}
 Let $d\in \N$, $G$ be a planar graph, and $\pleq$ be any Koebe ordering of $G$. Then $$\wcol_d(G,\pleq)\leq \Oh(d^4\ln d).$$
\end{theorem}

We do not know whether the bound provided by \cref{thm:wcol-ub} is asymptotically tight. More precisely, in \cref{sec:lower-bounds} we provide an example of a planar graph and its Koebe ordering whose weak $d$-coloring number is of the order $\Theta(d^3)$. This leaves a gap between the $\Omega(d^3)$ lower bound and the $\Oh(d^4\ln d)$ upper bound. In \cref{sec:lower-bounds} we discuss why closing this gap might be an interesting research direction.


\medskip

The remainder of this section is devoted to the proof of \cref{thm:wcol-ub}. On high level, the reasoning follows a general strategy employed in~\cite{DvorakPUY21}, but we tailor it to the setting of coin models in order to obtain improved bounds. 

Let us fix a coin model $D(\cdot)$ of the given graph $G$ using which the vertex ordering $\pleq$ was constructed. We need to show that for every vertex $u$ of $G$, we have 
\begin{equation}\label{eq:wcol-goal}
|\WReach_d^{G,\pleq}[u]|\leq \Oh(d^4\ln d). 
\end{equation}
Let us fix the vertex $u$ for the remainder of the proof. By scaling and translation, we may assume that $D(u)$ is the disc of radius $1$ with center $\zz$. Also, denote $W\coloneqq \WReach_d^{G,\pleq}[u]$ for brevity. Without loss of generality assume that $d>12$.

We partition the vertex set of $G$ into {\em{buckets}} $\{B_i\colon i\in \Z\}$ as follows: for $i\in \Z$, we set
$$B_i\coloneqq \{\, v\in V(G)~|~d^{3i}\leq r(D(v))<d^{3i+3}\,\},$$
where $r(D)$ is the radius of disc $D$. Clearly $u\in B_0$ and $W\subseteq \bigcup_{i\geq 0} B_i$. We first observe that an area argument based on the measure $\mu$ introduced in \cref{sec:adm} shows that every bucket contains only few vertices from $W$.

\begin{lemma}\label{lem:bucket-small}
 For every $i\geq 0$, we have
 $$|B_i\cap W|\leq \Oh(d^2\ln d).$$
\end{lemma}
\begin{proof}
 Consider any $w\in W\setminus \{u\}$. Let $r$ be the radius of $D(w)$ and $a$ be the distance from the center of $D(w)$ to $\zz$. Note that $w<u$, so $r\geq 1$. Let $P$ be a weak reachability path witnessing that $w\in \WReach_d[u]$. Observe that each $v\in V(P)$ satisfies $w\pleq v$, so the radius of $D(v)$ is not larger than $r$. Since $P$ has length at most $d$, we conclude that
 \begin{equation}\label{eq:giraffe}
 a\leq 1+(d-1)\cdot 2r+r\leq 2dr.
 \end{equation}
 On the other hand, as $D(w)$ has radius $r$ and is disjoint with $D(u)$, we have
 \begin{equation}\label{eq:elephant}
 a\geq 1+r.
 \end{equation}
 
 Suppose now that additionally $w\in B_i$ for some $i\geq 0$. Then $d^{3i}\leq r<d^{3i+3}$, implying by~\eqref{eq:giraffe} and~\eqref{eq:elephant} that $1+d^{3i}\leq a\leq 2d^{3i+4}$. Let $D'(w)$ be the disc with same center as $D(w)$ but twice smaller radius, that is, $r/2$. It follows that $D'(w)$ is entirely contained in the ring
 $$R\coloneqq \{x\in \R^2~|~d^{3i}/2\leq \|x\|\leq 3d^{3i+4}\}.$$
 Further, by \cref{cl:mu-lb} and \eqref{eq:giraffe}, we have
 $$\mu(D'(w))\geq \frac{\pi}{64}\cdot \frac{r^2}{d^2r^2}=\Omega\left(\frac{1}{d^2}\right).$$
 On the other hand, we have
 $$\mu(R)=\int_{d^{3i}/2}^{3d^{3i+4}} \frac{2\pi t}{t^2}\ dt=2\pi\left(\ln (3d^{3i+4})- \ln (d^{3i}/2)\right)=8\pi\ln d+\ln 6\in \Oh(\ln d).$$
 As discs $\{D'(w)\colon w\in B_i\cap W\}$ are pairwise disjoint, and at most one vertex of $B_i\cap W$ can be equal to $u$, we conclude that $|B_i\cap W|\leq \Oh(d^2\ln d)$.
\end{proof}

Therefore, by \cref{lem:bucket-small} to prove~\eqref{eq:wcol-goal} it suffices to show the following.

\begin{lemma}\label{lem:few-buckets}
There are $\Oh(d^2)$ nonnegative integers $i$ such that $B_i\cap W\neq \emptyset$.
\end{lemma}

For the remainder of this section we focus on proving \cref{lem:few-buckets}.

For two indices $j>i\geq 0$, call $j$ {\em{accessible}} from $i$ if there exists
a weak reachability path $P$ that starts at $u$, ends at a vertex of $B_j$, has
length at most $d$, and satisfies $V(P)\subseteq B_j\cup \bigcup_{k\leq i}
B_k$. The key observation towards the proof of \cref{lem:few-buckets} is provided by the following lemma, whose proof heavily relies on the geometry of the Euclidean plane.

\begin{lemma}\label{lem:jumps-empty}
 Suppose indices $j<j'$ are both accessible from $i\geq 0$. Then
 $$B_t\cap W=\emptyset \qquad\textrm{for each }t\in \{i+2,i+3,\ldots,j-1\}.$$
\end{lemma}
\begin{proof}
 Let $P$ and $P'$ be weak reachability paths witnessing that $j$ and $j'$ are accessible from $i$, respectively. By trimming $P$ if necessary we may assume that the endpoint $w$ of $P$ other than $u$ is the only vertex on $P$ that belongs to $B_j$, and all the other vertices of $P$ belong to $\bigcup_{k\leq i} B_k$. The same can be assumed about the endpoint $w'$ of $P'$ other than $u$.
 
 Denote
 $$\rho\coloneqq d^{3j},\qquad \xi\coloneqq d^{3i+6}\qquad \alpha\coloneqq 2d^{3i+4},\qquad \beta\coloneqq 2d^{3i+7}.$$
 Note that since $w\in B_j$ and $w'\in B_j'$, both disks $D(w)$ and $D(w')$ have radii at least $\rho$. Therefore, we can find a disk $D$ contained in $D(w)$ such that $D$ has radius exactly $\rho$ and $D$ touches the disk of the predecessor of $w$ on $P$. Similarly, we can find a disk $D'$ contained in $D(w')$ such that $D'$ has radius exactly $\rho$ and $D'$ touches the disk of the predecessor of $w'$ on $P'$. 
 
 \begin{figure}[ht]
 \centering
  \includegraphics[width=0.9\textwidth]{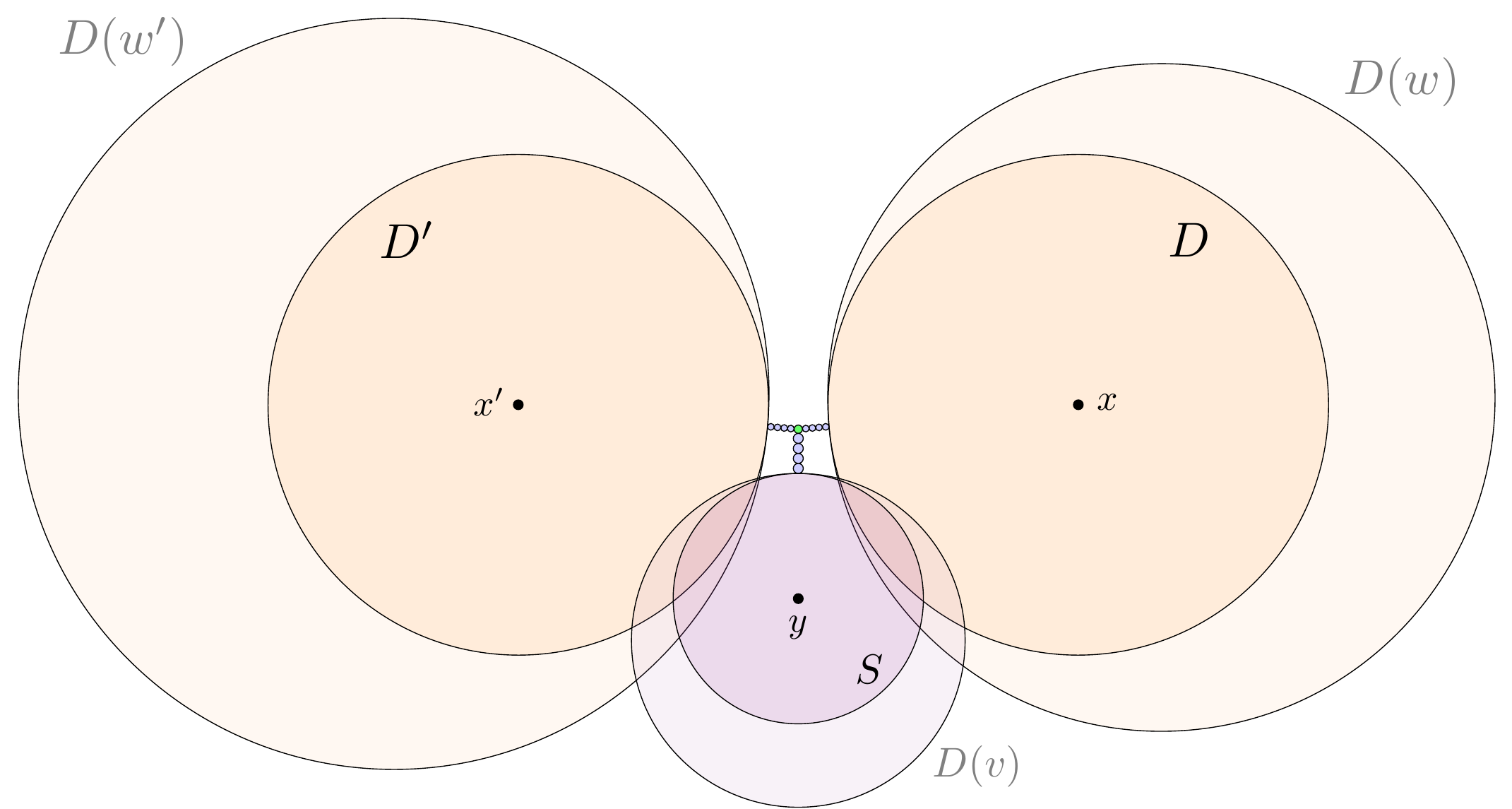}
  \caption{Situation in the proof of \cref{lem:jumps-empty}. Here is the geometric intuition behind the proof. $D$~and $D'$ are huge disks (each of radius $\rho$), which are nevertheless very close to $\zz$ (at distance at most $\alpha\ll \rho$). Therefore, $D$ and $D'$ necessarily create a ``corridor'' of width roughly $2\alpha$ into which all other disks close to $\zz$ must fit. However, the existence of a vertex $v\in W\cap \bigcup_{i+2\leq k <j} B_k$ would imply the existence of a disk $S$, disjoint from $D$ and $D'$, whose radius is much larger than $2\alpha$, and whose distance from~$\zz$ is significantly smaller than $\rho$. Then $S$ cannot fit into the corridor without intersecting $D$ or $D'$, a~contradiction.}\label{fig:squeezing}
 \end{figure}

 Let $x$ and $x'$ be the centers of $D$ and $D'$, respectively.
 Since all vertices on $P$ except for $w$ belong to $\bigcup_{k\leq i} B_k$, the radii of the disks associated with them are smaller than $d^{3i+3}$. It follows that the common point of $D$ and the disk of the predecessor of $w$ on $P$ is at distance at most $1+2(d-1)d^{3i+3}\leq \alpha$ from~$\zz$. By the triangle inequality we conclude that
 \begin{equation}\label{eq:turtle1}
  \|x\|\leq \rho+\alpha.
 \end{equation}
 Analogous reasoning for the disk $D'$ yields that 
 \begin{equation}\label{eq:turtle2}
  \|x'\|\leq \rho+\alpha.
 \end{equation}
 On the other hand, $D$ and $D'$ are respectively contained in disks $D(w)$ and $D(w')$, which have disjoint interiors. (Note here that $w\neq w'$, because $j\neq j'$ and $w\in B_j$ and $w'\in B_{j'}$.) Therefore, $D$ and $D'$ have disjoint interiors, implying that
 \begin{equation}\label{eq:snail}
  \|x-x'\|\geq 2\rho.
 \end{equation}
 
 Suppose now, aiming at a contradiction, that there exists a vertex $v\in W$ that belongs to $W_t$ for some $t\in \{i+2,i+3,\ldots,j-1\}$. We may choose such $v$ so that $t$ is minimum possible and, subject to this, the minimum length of a weak reachability path from $u$ to $v$ is also minimum possible. Thus, if $Q$ is a minimum length weak reachability path from $u$ to $v$, then $Q$ has length at most $d$ and all vertices of $Q$ except for $v$ belong to $\bigcup_{k\leq i+1} B_k$.
 
 Since $v\in B_t$ and $t\geq i+2$, the radius of $D(v)$ is at least $\xi=d^{3i+6}$. Therefore, we can find a disk $S$ entirely contained in $D(v)$ so that $S$ has radius exactly $\xi$ and $S$ touches the disk of the predecessor of $v$ on $Q$. Since all vertices on $Q$ except for $v$ belong to $\bigcup_{k\leq i+1} B_k$, the disks associated with them have radii smaller than $\xi$. Therefore, the common point of $S$ and the disk of the predecessor of $v$ on $Q$ is at distance at most $1+2(d-1)\xi$ from $\zz$. Denoting the center of $S$ by $y$, by triangle inequality we again conclude that
 \begin{equation}\label{eq:worm}
\|y\|\leq 1+2(d-1)\xi+\xi\leq 2d\xi=\beta.  
 \end{equation}
 Note that $v$ is different from $w$ and $w'$, since $v\in B_t$ and $t<j<j'$. So $D(v)$ and $D(w)$ have disjoint interiors, implying that $S$ and $D$ have disjoint interiors; similarly for $S$ and $D'$. We conclude that
 \begin{equation}\label{eq:snake}
  \|x-y\|\geq \rho+\xi\qquad\textrm{and}\qquad \|x'-y\|\geq \rho+\xi.
 \end{equation}
 
 Now our goal is to combine inequalities~\eqref{eq:turtle1},~\eqref{eq:turtle2},~\eqref{eq:snail},~\eqref{eq:worm} and~\eqref{eq:snake} in order to obtain a contradiction. While the argument that follows might seem to consist of soulless algebraic manipulations, there is a clear geometric intuition behind it; see the caption of \cref{fig:squeezing}.
 
 First, observe that, by~\eqref{eq:turtle1},~\eqref{eq:turtle2} and~\eqref{eq:snail},
 \begin{equation}\label{eq:frog}
\|x+x'\|^2 = 2\|x\|^2+2\|x'\|^2-\|x-x'\|^2\leq 4(\rho+\alpha)^2-4\rho^2=8\rho\alpha+4\alpha^2.  
 \end{equation}
 On the other hand, from~\eqref{eq:snake} we infer that
 \begin{equation}\label{eq:salamander}
\|x-y\|^2+\|x'-y\|^2\geq 2(\rho+\xi)^2.
 \end{equation}
 However, observe that
 \begin{eqnarray*}
\|x-y\|^2+\|x'-y\|^2 & = &\|x\|^2+\|x'\|^2+2\|y\|^2-2\sca{x+x'}{y}\\
& \leq & \|x\|^2+\|x'\|^2+2\|y\|^2+2\|x+x'\|\cdot \|y\| \\
& \leq & 2(\alpha+\rho)^2+2\beta^2+4\beta\sqrt{2\rho\alpha+\alpha^2},
 \end{eqnarray*}
 where in the first step we used the Cauchy-Schwartz inequality and the second step follows from~\eqref{eq:turtle1},~\eqref{eq:turtle2},~\eqref{eq:worm} and~\eqref{eq:frog}. Combining this with~\eqref{eq:salamander} yields
 $$2(\rho+\xi)^2\leq 2(\alpha+\rho)^2+2\beta^2+4\beta\sqrt{2\rho\alpha+\alpha^2},$$
 which readily reduces to
 \begin{equation}\label{eq:shrimp}
 2\rho \xi + \xi^2 \leq 2\rho\alpha+\beta^2+\alpha^2+2\beta\sqrt{2\rho\alpha+\alpha^2}.
 \end{equation}
 
 We may assume $j\geq i+3$, for otherwise the lemma statement holds vacuously.
 Hence, we have
 \begin{eqnarray*}
2\rho\alpha+\beta^2+\alpha^2+2\beta\sqrt{2\rho\alpha+\alpha^2} & \leq & 2\rho\alpha+\beta^2+\alpha^2+4\beta\sqrt{\rho\alpha} \\
& = & 4d^{3j+3i+4}+4d^{6i+14}+4d^{6i+8}+8\sqrt{2}\cdot d^{\frac{3}{2}j+\frac{9}{2}i+9}\\ & \leq & 24d^{3j+3i+5},
 \end{eqnarray*}
 where in the last step we note that the value $3j+3i+5$ is never smaller than any of the exponents of the involved summands.
 On the other hand, we have
 $$2\rho \xi + \xi^2\geq 2\rho\xi=2d^{3j+3i+6}.$$
 However, as we assumed $d>12$, we have $2d^{3j+3i+6}>24d^{3j+3i+5}$. This is a contradiction with~\eqref{eq:shrimp} and the proof is complete.
\end{proof}

Intuitively, our goal now is to perform a two-level greedy construction after which \cref{lem:jumps-empty} will be applicable.
Define indices $i_0,i_1,\ldots,i_p$ inductively as follows.
\begin{itemize}[nosep]
 \item $i_0=0$, and
 \item for $t\geq 0$, $i_{t+1}$ is the maximum index accessible from $i_{t}$. In case there is no such index, the construction finishes without defining $i_{t+1}$; that is, we set $p\coloneqq t$.
\end{itemize}
We observe the following.

\begin{lemma}\label{lem:major-greedy-dist}
 Let $w\in \WReach_t[u]$ for some $0\leq t\leq d$, and let $i$ be such that $w\in B_i$. Then $i\leq i_t$.
\end{lemma}
\begin{proof}
 We proceed by induction on $t$, with the base case for $t=0$ being trivial. Assume then that $t\geq 1$.
 Let $P$ be a weak reachability path witnessing that $w\in \WReach_t[u]$. Let $w'$ be the $\pleq$-maximum vertex among $V(P)\setminus \{w\}$, and let $i'$ be such that $w'\in B_{i'}$. The prefix of $P$ from $u$ to $w'$ witnesses that $w'\in \WReach_{t-1}[u]$. By induction, we have $i'\leq i_{t-1}$. Note that either $i=i'$, or $P$ witnesses that $i$ is accessible from~$i'$. Together with $i'\leq i_{t-1}$ this implies that $i\leq i_t$.
\end{proof}

From \cref{lem:major-greedy-dist} we can immediately infer the following.

\begin{lemma}\label{lem:major-steps}
 It holds that $p\leq d$.
\end{lemma}
\begin{proof}
 Suppose otherwise, that $p>d$.
 By the definition of $i_p$, there exists $w\in B_{i_p}\cap W=B_{i_p}\cap \WReach_d[u]$. Then \cref{lem:major-greedy-dist} applied to $w$ implies that $i_p\leq i_d$. But $i_0,i_1,i_2,\ldots$ is a strictly increasing sequence, a contradiction.
\end{proof}

For a fixed $t\in \{0,1,\ldots,p-1\}$, we define indices $i_{t,0},i_{t,1},\ldots,i_{t,p_t}$ similarly as before:
\begin{itemize}[nosep]
 \item $i_{t,0}=i_t$, and
 \item for $s\geq 0$, $i_{t,s+1}$ is the maximum index that is accessible from $i_{t,s}$ and smaller than $i_{t+1}$. In case there is no such index, the construction finishes without defining $i_{t,s+1}$; that is, we set $p_t\coloneqq s$.
\end{itemize}
We observe the following.

\begin{lemma}\label{lem:minor-greedy-dist}
 Let $w\in \WReach_r[u]$ for some $0\leq r\leq d$, and let $i$ be such that $w\in B_i$.
 Suppose that $i_t\leq i <i_{t+1}$ for some $t\in \{0,1,\ldots,p-1\}$. Then $i\leq i_{t,r-t}$.
\end{lemma}
\begin{proof}
 Note that \cref{lem:major-greedy-dist} implies that $r\geq t$. Moreover, if $r=t$ then we necessarily have $i=i_t=i_{t,0}$. This establishes the base case for induction on $r-t$. The induction step is essentially identical to the one from the proof of 
 \cref{lem:major-greedy-dist}; we leave the details to the reader. 
\end{proof}

So similarly as in \cref{lem:major-steps}, we obtain the following.

\begin{lemma}\label{lem:minor-steps}
For every $t\in \{0,1,\ldots,p-1\}$, we have $p_t\leq d-t$.
\end{lemma}

Finally, we can use \cref{lem:jumps-empty} to argue the following.

\begin{lemma}\label{lem:twice-covered}
 Let $t\in \{0,1,\ldots,p-1\}$ and $s\in \{0,1,\ldots,p_t-1\}$. Then there is at most one index $i$ such that
 $$i_{t,s}<i<i_{t,s+1}\qquad\textrm{and}\qquad B_i\cap W\neq \emptyset.$$
\end{lemma}
\begin{proof}
 By construction, $i_{t+1}$ and $i_{t,s+1}$ are two different indices that are both accessible from $i_{t,s}$. Then \cref{lem:jumps-empty} implies that every index $i$ satisfying the condition in the lemma statement must be equal to $i_{t,s}+1$, so there can be at most one such index.
\end{proof}

We can now conclude the proof of \cref{lem:few-buckets}.

\begin{proof}[Proof of \cref{lem:few-buckets}]
 Let $L$ be the set of all indices $i$ satisfying $B_i\cap W\neq \emptyset$. By construction, we have
 $$L\subseteq \{i_0,i_1,\ldots,i_p\}\cup \bigcup_{t\in \{0,1,\ldots,p-1\}} \{i_t+1,i_t+2,\ldots,i_{t,p_t}\}.$$
 By \cref{lem:minor-steps} and \cref{lem:twice-covered}, for each $t\in \{0,1\ldots,p-1\}$ we have
 $$|L\cap \{i_t+1,i_t+2,\ldots,i_{t,p_t}\}|\leq 2p_t\leq 2(d-t).$$
 So by \cref{lem:major-steps}, we conclude that
 $$|L|\leq (d+1)+2\sum_{t=0}^{d-1} (d-t)=(d+1)+d(d+1)=(d+1)^2.\qedhere$$
\end{proof}

As argued, \cref{lem:bucket-small} together with \cref{lem:few-buckets} prove \cref{thm:wcol-ub}.

\section{Lower bounds for Koebe orderings}\label{sec:lower-bounds}

In this section we discuss some lower bounds for generalized coloring numbers of Koebe orderings of planar graphs.
We start with a very simple lower bound that witnesses the tightness of \cref{thm:scol-ub}. The construction can be considered folklore, so we include it for completeness and because it will be used as a building block for a later construction.
 
 \begin{figure}[ht]
 \centering
  \includegraphics[width=0.4\textwidth]{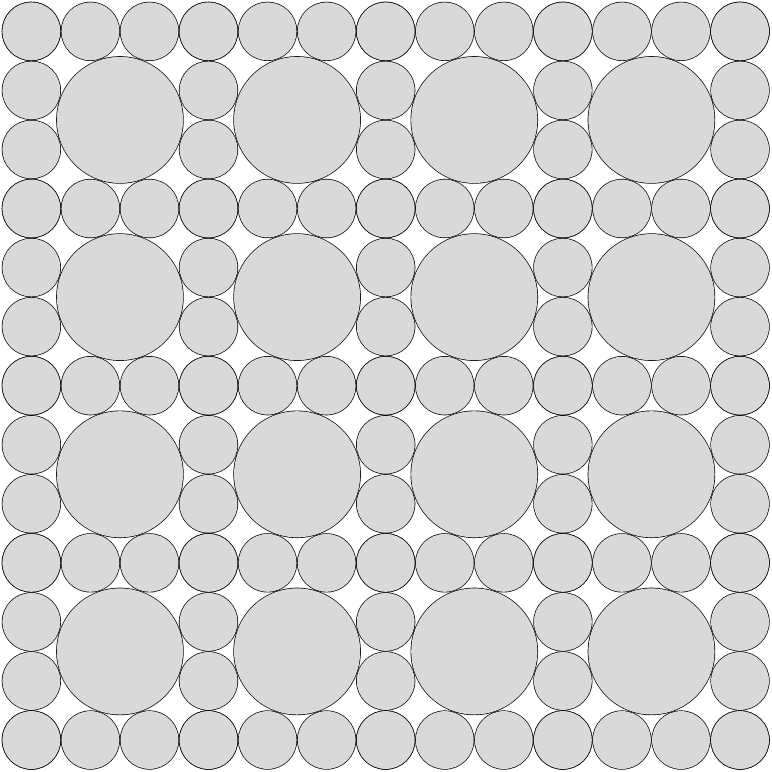}
  \caption{The construction of \cref{prop:scol-lb} for $d=26$. Every small disc has radius $1$, every large disc has radius $\sqrt{10}-1$.}\label{fig:grid}
 \end{figure}

\begin{proposition}\label{prop:scol-lb}
 For every $d\in \N$ there exists a planar graph $G$ and a coin model $D(\cdot)$ for $G$ such that for any Koebe ordering $\pleq$ of $G$ constructed with respect to $D(\cdot)$ we have $\scol_d(G,\pleq)\geq \Omega(d^2)$.
\end{proposition}
\begin{proof}
 Without loss of generality assume that $d$ is congruent to $2$ modulo $12$. Consider the intersection graph $G$ of discs arranged as in~\cref{fig:grid}: it is a $d/2\times d/2$ grid of unit-radius discs, where all quadruples of discs with adjacent indices congruent to $2$ or $0$ in the natural indexing are replaced with single discs of radius $\sqrt{10}-1$.

 Let $\pleq$ be any Koebe ordering of $G$ constructed for the particular coin model described above. Note that $\pleq$ places all small discs (those of radius $1$) after all large discs (those of radius $\sqrt{10}-1$). Therefore, if $D$ is a small disc that is the smallest in $\pleq$, then it is easy to see that every large disc is strongly $d$-reachable from $D$. Since the number of large discs is $\Omega(d^2)$, it follows that $\scol_d(G,\pleq)\geq \Omega(d^2)$
\end{proof}

Next, we provide a lower bound showing that in \cref{thm:wcol-ub} one cannot obtain a better bound than cubic in $d$.\footnote{We note that we were informed by Piotr Micek~\cite{micek-private} that an asymptotically same lower bound can be obtained when considering Koebe orderings of natural coin models of {\em{stacked triangulations}} (or {\em{Apollonian networks}}), but this example seems somewhat harder to analyze formally.}

 \begin{figure}[ht]
 \centering
  \includegraphics[width=0.4\textwidth]{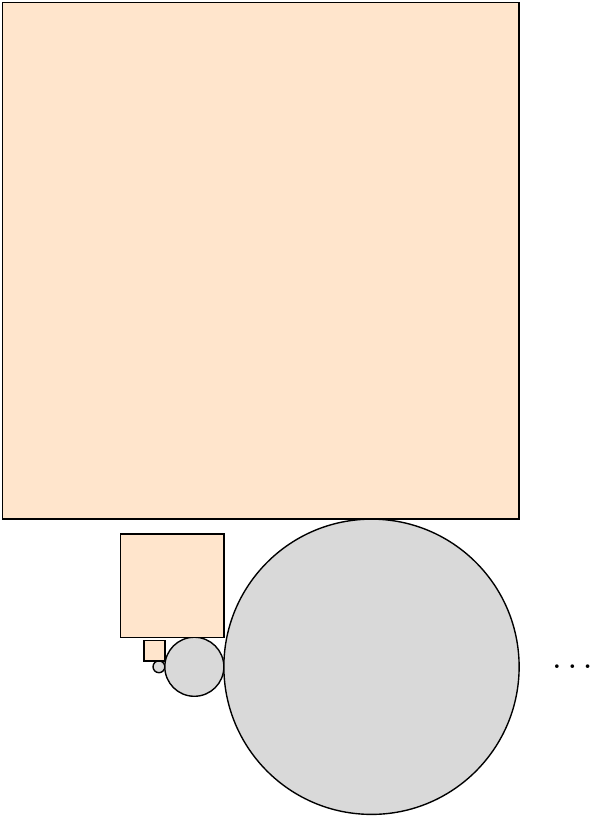}
  \caption{The construction of \cref{prop:wcol-lb}. Every orange box is an appropriately scaled construction from \cref{fig:grid}. For clarity of presentation the figure is somewhat not to scale: every orange box should be exactly $d/4$ times wider than the grey disc it is tangent to.}\label{fig:multigrid}
 \end{figure}

\begin{proposition}\label{prop:wcol-lb}
 For every $d\in \N$ there exists a planar graph $G$ and a coin model $D(\cdot)$ for $G$ such that for any Koebe ordering $\pleq$ of $G$ constructed with respect to $D(\cdot)$ we have $\wcol_d(G,\pleq)\geq \Omega(d^3)$.
\end{proposition}
\begin{proof}
 Without loss of generality we assume that $d$ is congruent to $4$ modulo $24$.
 
 We construct a planar graph $G$ by specifying its coin model $D(\cdot)$. Consider first the following construction of a {\em{gadget}}. First, apply the grid-like construction from \cref{fig:grid} where the ``grid'' has $d/4$ unit-radius discs along each side. Then, add one unit-radius disc that touches the disc in the bottom-right corner of the grid at its bottom-most point (i.e. the one with the lowest second coordinate). This additional disc will be called the {\em{interface}} of the gadget. In \cref{fig:multigrid}, every orange box together with the tangent blue disc represents a single gadget, where the blue disc is the interface of the gadget.

 Finally, construct $G$ together with its coin model $D(\cdot)$ by taking $d/2$ gadgets, where the $i$th gadget is scaled by a factor of $d^{i-1}$, and arranging them so that the interfaces form a horizontally aligned sequence of discs with increasing radii; see \cref{fig:multigrid}. It is easy to see that discs from different gadgets have pairwise non-intersecting interiors. Hence the intersection graph of the discs is a planar graph $G$, and the constructed discs form a coin model of $G$.
 
 Let $\pleq$ be any Koebe ordering of $G$ constructed with respect to the coin model described above.
 Let $D$ be the interface of the first gadget. Observe that for every $i\in \{1,\ldots,d/2\}$, every large disc within the grid in the $i$th gadget is weakly $d$-reachable from $D$ in $\pleq$. Indeed, it suffices to first pass from the first interface to the $i$th interface along a path of length $i-1$ consisting of consecutive interfaces, and then reach the considered large disc of the grid by a path of length at most $d/2$ whose internal vertices are small discs of the grid. Since the grid in every gadget contains $\Omega(d^2)$ large discs and there are $d/2$ gadgets, it follows that $D$ weakly $d$-reaches $\Omega(d^3)$ other discs. So $\wcol_d(G,\pleq)\geq \Omega(d^3)$.
\end{proof}

Let us analyze the construction of \cref{prop:wcol-lb} through the lenses of the proof of \cref{thm:wcol-ub}, where we consider $u$ to be the vertex corresponding to the interface of the first gadget. Then in the notation of the said proof, there are $\Theta(d)$ buckets $B_i$ satisfying $B_i\cap W\neq \emptyset$, and each of them contains $\Theta(d^2)$ vertices that are weakly $d$-reachable from $u$. Thus, the bound from \cref{lem:bucket-small} is almost tight --- up to a logarithmic factor --- while the bound from \cref{lem:few-buckets} is not: there are only $\Oh(d)$ reachable buckets, compared to the $\Oh(d^2)$ upper bound provided by \cref{lem:few-buckets}. It is possible to construct another example where the number of reachable buckets is $\Theta(d^2)$, but then each of them contains only $\Theta(d)$ weakly $d$-reachable vertices. (We refrain from giving a formal exposition of this example for the sake of brevity.) We suspect this might not be a coincidence, and we actually conjecture that for every $d\in \N$, planar graph $G$, and Koebe ordering $\pleq$ of $G$, it holds that $\wcol_d(G,\pleq)\leq d^3\cdot \ln^{\Oh(1)} d$. If this was the case, then it would be conceivable that by taking any Koebe ordering of a planar graph, and somehow reshuffling similarly-sized discs in order to avoid the example from \cref{prop:scol-lb}, it would be possible to obtain a vertex ordering with weak $d$-coloring number that is subcubic in $d$. This would resolve a notorious open problem in the area, see e.g.~\cite[Problem~1]{JoretM21}.

\paragraph*{Acknowledgements.} The results presented in this paper were obtained during the trimester on Discrete Optimization at the Hausdorff Research Institute for Mathematics (HIM) in Bonn, Germany. We are thankful for the possibility of working in the stimulating and creative research environment at HIM. We also thank Piotr Micek for helpful discussions about the state-of-the-art of the bounds on generalized coloring numbers in planar graphs.

\bibliography{ref}

\appendix

\section{Lower bounds for admissibility and strong coloring number}\label{app:scol} 

In this appendix we present two lower bounds: for the $d$-admissibility and the strong $d$-coloring number of planar graphs. Both constructions have been known, but were either unpublished or had the analysis omitted. We include them here for completeness and to provide a source for future reference.

We start with $d$-admissibility. The following lower bound was communicated to us by Zden\v{e}k Dvo\v{r}\'ak and Sebastian Siebertz~\cite{adm-private} and has not been published. We are grateful to Zden\v{e}k and Sebastian for allowing us to include their construction here.

\newcommand{\Ns}{\mathsf{N}}
\newcommand{\Es}{\mathsf{E}}
\newcommand{\Ws}{\mathsf{W}}
\newcommand{\NW}{\mathsf{NW}}
\newcommand{\NE}{\mathsf{NE}}
\newcommand{\SW}{\mathsf{SW}}
\newcommand{\SE}{\mathsf{SE}}

\begin{proposition}\label{prop:adm-lb}
 For every integer $k\geq 2$ there exists a planar graph $G$ such that $\adm_{k 2^{k+2}}(G)\geq 2^k-1$.
\end{proposition}
\begin{proof}
 The construction is depicted in \cref{fig:construction}.
 Let $L=\{\Ws,\Es\}^k$ and $K=\{\Ws,\Es\}^{<k}$ be the sets of all words over the alphabet $\{\Ws,\Es\}$ of length $k$ and of length strictly smaller than $k$, respectively.
 For each $w\in K$, construct a $2^k\times 2^k$ grid $H_w$ and denote  the sides of $H_w$ as $\NW_w$, $\NE_w$, $\SE_w$, and $\SW_w$ in order. Whenever speaking about the order of vertices on these sides, we use the north-to-south convention: the vertices are ordered naturally along the sides, where the first vertex on sides $\NW_w$ and $\NE_w$ is the corner at their intersection, and the last vertex on sides $\SW_w$ and $\SE_w$ is the corner at their intersection.
 
 Next, whenever $w,w'\in K$ are such that $w'=w\Ws$, for each $i\in \{1,\ldots,2^k\}$ we add an edge between the $i$th vertex of $\SW_w$ and the $i$th vertex of $\NE_{w'}$. In case $w'=w\Es$, perform a symmetric construction but with the roles of $\Ws$ and $\Es$ swapped. Finally, for each $u\in L$ construct a vertex $v_u$ and make it adjacent to all vertices of $\SW_w$ if $u=w\Ws$ for some $w\in K$, and to all vertices of $\SE_w$ if $u=w\Es$ for some $w\in K$. As depicted in \cref{fig:construction}, the graph $G$ constructed in this way is planar.

  \begin{figure}[ht]
 \centering
  \includegraphics[width=\textwidth]{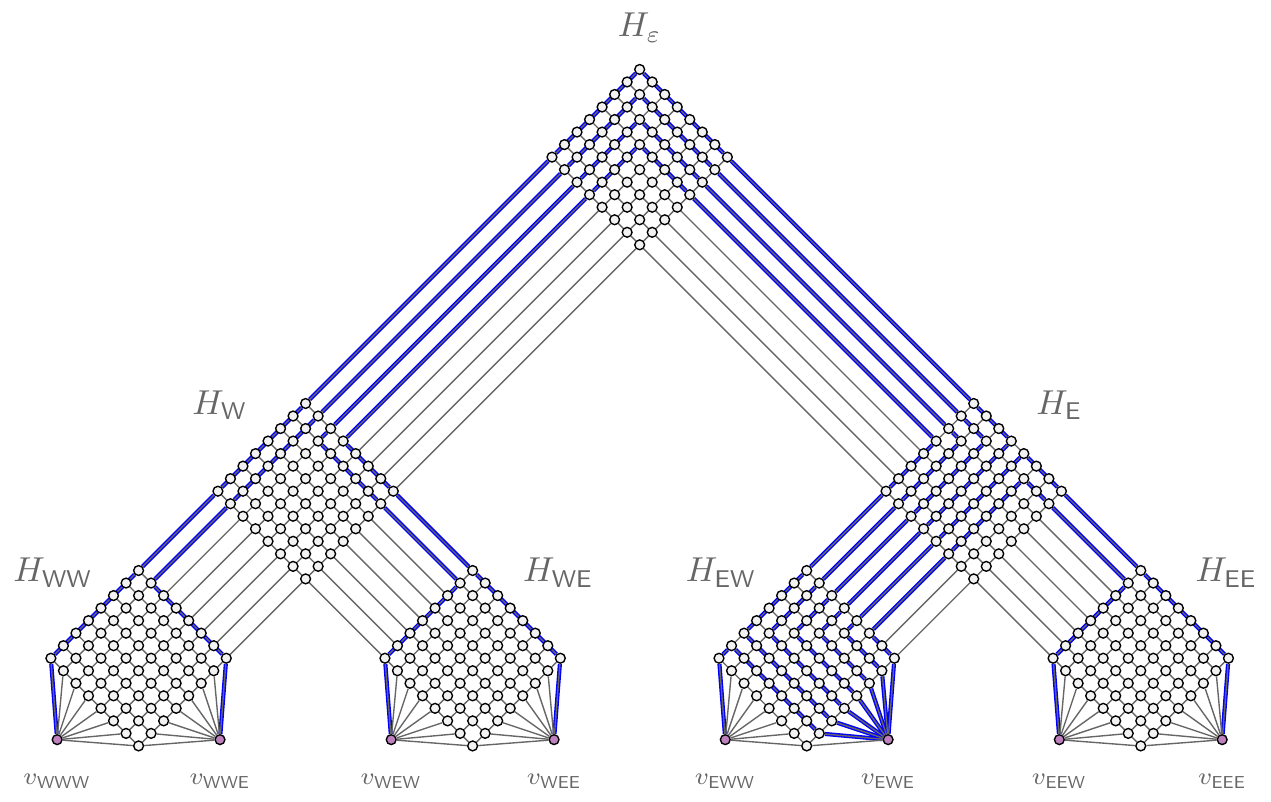}
  \caption{Construction of graph $G$ for $k=3$. 
  Path family $\Qq_{\Es\Ws\Es}$ constructed later in the proof is highlighted in blue.}\label{fig:construction}
 \end{figure}

 Call a path $P$ in $G$ {\em{straight}} if for every $w\in K$, the intersection of $P$ with $H_w$ is either empty or consists of a single path of length at most $2^{k+1}-2$. Note that thus, a straight path intersects at most $2k-3$ different grids $H_w$ and may have vertices of $\{v_u\colon u\in L\}$ only as endpoints. Thus, every straight path has length at most $(2k-3)(2^{k+1}-1)+1\leq k2^{k+2}$.

 We first construct, for every $w\in K$ and $s\in \{\Ws,\Es\}$, a family $\Pp_{ws}$ of paths in $G$ with the following properties:
 \begin{itemize}[nosep]
  \item paths in $\Pp_{ws}$ are straight and pairwise vertex-disjoint;
  \item $|\Pp_{ws}|=2^{k-|w|-1}$, where $|w|$ is the length of $w$;
  \item every path $P\in \Pp_{ws}$, starts at a vertex $v_u$ for some $u\in L$ such that $ws$ is a prefix of $u$, and ends at a vertex that is among the first $|\Pp_{ws}|$ vertices of the side $(\mathsf{S} s)_{w}$; and 
  \item for every $P\in \Pp_{ws}$, all internal vertices of $P$ are contained in the union of grids $H_{w'}$ for $w'\in K$ such that $ws$ is a prefix of $w'$.
 \end{itemize}
 Note that there are exactly $2^{k-|w|-1}$ vertices $v_u$ for which $ws$ is a prefix of $u$, hence every such vertex is an endpoint of a path from $\Pp_{ws}$. Similarly, every vertex among the first $|\Pp_{ws}|$ vertices of the side $(\mathsf{S} s)_{w}$ is an endpoint of a path from $\Pp_{ws}$.
 
 The construction proceeds by induction on $k-|w|$. In case $k-|w|=1$, we have $ws\in L$, so the side $(\mathsf{S} s)_{w}$ is entirely adjacent to the vertex $v_{ws}$. So it suffices to set $\Pp_{ws}=\{P\}$, where $P$ is the two-vertex path induced by $v_{ws}$ and the first vertex of the side $(\mathsf{S} s)_{w}$. For the induction step, by symmetry assume $s=\Ws$. By induction we can construct suitable families $\Pp_{w\Ws\Ws}$ and $\Pp_{w\Ws\Es}$. To construct $\Pp_{w\Ws}$, it suffices to take the union of $\Pp_{w\Ws\Ws}$ and $\Pp_{w\Ws\Es}$, then extend each path of this union within the grid $H_{w\Ws}$ so that it ends at a different vertex among the first $2^{k-|w|+1}$ vertices of the side $\NE_{w\Ws}$, and finally extend it by a single edge so that it ends at a different vertex among the first $2^{k-|w|+1}$ vertices of the side $\SW_{w}$. It can be easily seen that this can be done; see the left panel of \cref{fig:patching}. Moreover, each extension uses a path within $H_{w\Ws}$ of length at most $2^{k+1}-2$ plus one edge connecting $H_{w\Ws}$ and $H_w$. Thus, the paths of $\Pp_{w\Ws}$ remain straight. 
 
 This concludes the inductive construction of path families $\Pp_{ws}$ for $w\in K$ and $s\in \{\Ws,\Es\}$.

  \begin{figure}[ht]
 \centering
  \includegraphics[width=\textwidth]{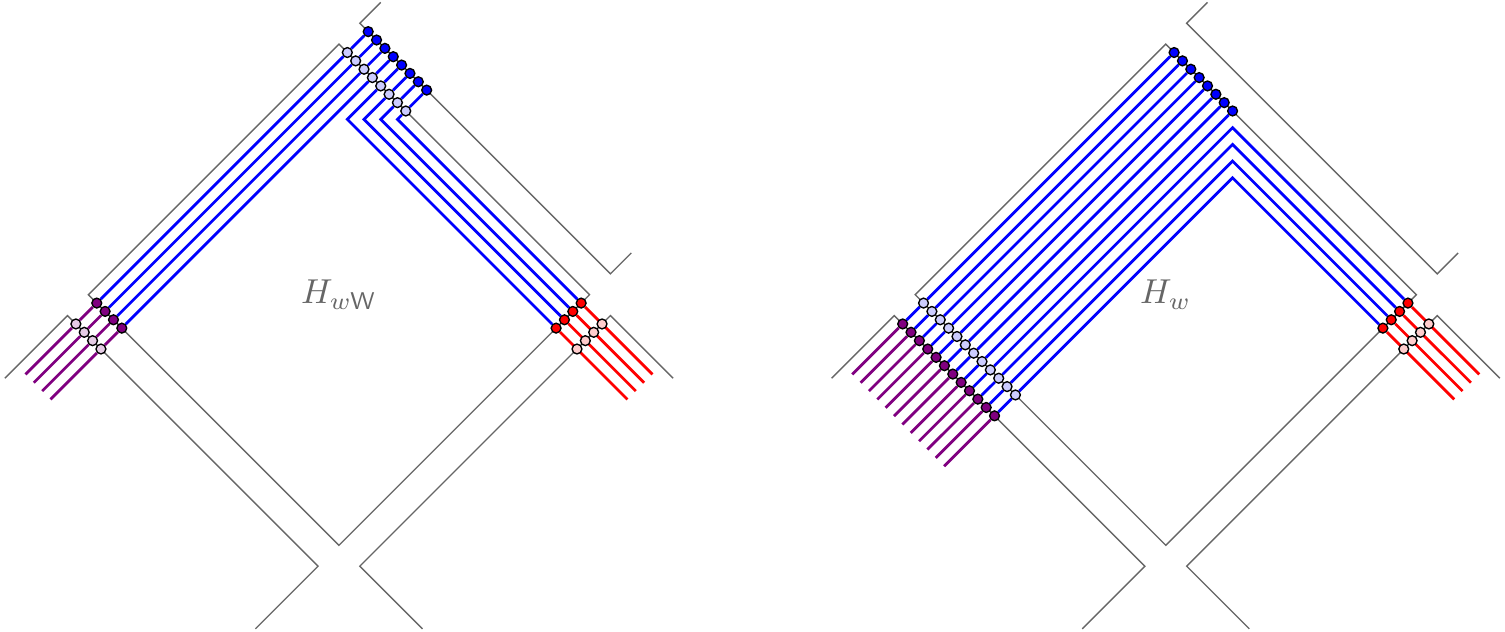}
  \caption{Left panel: Induction step in the construction of $\Pp_{w\Ws}$. Families $\Pp_{w\Ws\Ws}$ and $\Pp_{w\Ws\Es}$ are depicted in violet and red, respectively, while paths used in the extension are depicted in blue. Right panel: Induction step in the construction of $\Qq_{w,u}$ when the last symbol of $w$ is $\Ws$ and $u$ is a prefix of $w\Ws$. Families $\Qq_{w\Ws,u}$ and $\Pp_{w\Ws\Es}$ are depicted in violet and red, respectively, while paths used in the extension are depicted in blue.}\label{fig:patching}
 \end{figure}

 Next, we construct, for every $u\in L$ and every $w\in K$ that is a prefix of $u$, a family of paths $\Qq_{w,u}$ with the following properties:
 \begin{itemize}[nosep]
  \item paths in $\Qq_{w,u}$ are straight and pairwise vertex-disjoint except for sharing $v_u$;
  \item $|\Qq_{w,u}|=2^k-1$;
  \item every path $P\in \Qq_{w,u}$ starts at $v_u$ and ends at a vertex from the set $I_{w,u}$, where $I_{w,u}$ comprises the first $2^{k}-2^{k-|w|}$ vertices of the side $(\Ns t)_{w}$ ($t\in \{\Ws,\Es\}$ being the symbol other than the last symbol of $w$), and all vertices $v_{u'}$ such that $u'\in L$, $u'\neq u$, and $w$ is a prefix of $u'$; and
  \item for every $P\in \Qq_{w,u}$, all the internal vertices of $P$ are contained in the union of grids $H_{w'}$ for $w'\in K$ such that $w$ is a prefix of $w'$.
 \end{itemize}
 Note that in the third point above, $t$ is not well-defined in case $w$ is the empty word $\varepsilon$. But then $2^k-2^{k-|w|}=0$, so we simply do not include any vertices of this type in $I_{w,u}$.
 As before, observe that $|I_{w,u}|=2^k-1$, so every vertex of $I_{w,u}$ is an endpoint of a path from $\Qq_{w,u}$.
 
 Again, the construction proceeds by induction on $k-|w|$. A bit unconventionally, let us explain first the induction step, as the base case will follow from applying the same construction to a degenerate case. We also focus on the case when $w\neq \varepsilon$; the case $w=\varepsilon$ works analogously.
 Then, by symmetry, suppose that $w=w'\Ws$ for some $w'\in K$ and that $|w|<k-1$. Suppose further that $w\Ws$ is a prefix of~$u$, the other case (when $w\Es$ is a prefix of $u$) being again symmetric. Consider families $\Qq_{w\Ws,u}$ (obtained from the induction assumption) and $\Pp_{w\Es}$ (constructed before). Then, these families can be used to construct a suitable family $\Qq_{w,u}$ as follows (see the right panel of \cref{fig:patching}):
 \begin{itemize}[nosep]
  \item Start with setting $\Qq_{w,u}\coloneqq \Qq_{w\Ws,u}$.
  \item For each $Q\in \Qq_{w\Ws,u}$ that ends among the first $2^k-2^{k-|w|}$ vertices of the side $\NE_{w\Ws}$, say at the $i$th vertex, extend $Q$ using an edge between $H_{w\Ws}$ and $H_w$ and a path within $H_w$ so that it ends at the $i$th vertex of $\NE_{w}$.
  \item For each $Q\in \Qq_{w\Ws,u}$ that ends among the next $2^{k-|w|-1}$ vertices of the side  $\NE_{w\Ws}$, say at the vertex number $2^k-2^{k-|w|}+i$ on this side for some $i\in \{1,\ldots,2^{k-|w|-1}\}$, extend $Q$ using an edge between $H_{w\Ws}$ and $H_w$ and a path within $H_w$ so that it ends at the $i$th vertex of $\SE_{w}$. Then concatenate $Q$ with the unique path of $\Pp_{w\Es}$ that ends at the same vertex.
 \end{itemize}
 It can be easily seen that the extensions above can be obtained using paths within $H_w$ that are pairwise vertex-disjoint and of length at most $2^{k+1}-2$ each. Thus, the paths in the constructed family $\Qq_{w,u}$ remain straight and pairwise vertex-disjoint except for sharing $v_u$.
 
 The construction for $w=\varepsilon$ follows the same principle, except that we do not need to construct paths that end at vertices of $\NW_\varepsilon$ or $\NE_\varepsilon$. In the base case we may use the same construction, only that we interpret $\Qq_{u,u}$ to be a family of $2^k-1$ single-vertex paths that start and end at $v_u$.
 
 This concludes the inductive construction of path families $\Qq_{w,u}$ for $u\in L$ and $w\in K$ such that $w$ is a prefix of $u$. For $u\in L$, denote $\Qq_u=\Qq_{\varepsilon,u}$. Let us summarize the properties of $\Qq_u$:
 \begin{itemize}[nosep]
  \item Paths in $\Qq_u$ are pairwise vertex-disjoint except for sharing $v_u$.
  \item Each path in $\Qq_u$ is straight, and therefore of length at most $k2^{k+2}$.
  \item For every $u'\in L$, $u'\neq u$, there is a path $Q\in \Qq_u$ that connects $v_u$ with $v_{u'}$.
 \end{itemize}
 Given these properties, we can argue that $\adm_{k2^{k+2}}(G)\geq 2^k-1$. Consider any vertex ordering $\pleq$ of $G$. Let $u_{\max}\in L$ be such that $v_{u_{\max}}$ is $\pleq$-maximum among vertices $\{v_u\colon u\in L\}$. Then the path family $\Qq'_u$, obtained from $\Qq_u$ by trimming every path to the prefix till the first encounter of a vertex smaller in $\pleq$ than $v_{u_{\max}}$, witnesses that $\adm^{G,\pleq}_{k2^{k+2}}(v_{u_{\max}})\geq 2^{k}-1$. As $\pleq$ was chosen arbitrarily, it follows that $\adm_{k2^{k+2}}(G)\geq 2^k-1$.
\end{proof}

Note that from \cref{prop:adm-lb} it follows that for every $d\in \N$ there is a planar graph $G$ with $\adm_d(G)\geq \Omega(d/\ln d)$. This means that the upper bound of \cref{thm:adm-ub} is asymptotically tight.

\medskip

We continue with the strong $d$-coloring numbers.
As noted by van~den~Heuvel et al. in~\cite{HeuvelMQRS17}, there are planar graphs with strong $d$-coloring number $\Omega(d)$ and this lower bounds is actually realized by grids, but the work~\cite{HeuvelMQRS17} does not contain any formal proof of this fact. The argument can be considered folklore in the community, but we were unable to find any published work containing its presentation.

\begin{proposition}
 Let $G$ be the $d\times d$ grid for any $d\in \N$. Then $\scol_{3d-2}(G)\geq d/2$.
\end{proposition}
\begin{proof}
 Let $\pleq$ be any vertex ordering of $G$; our goal is to prove that $\scol_{3d-2}(G,\pleq)\geq d/2$. Index rows and columns of $G$ naturally. For each $i\in \{1,\ldots,d\}$, let $u_i$ be the $\pleq$-minimal vertex of the $i$th column, and let $U\coloneqq \{u_i\colon i\in \{1,\ldots,d\}\}$. 
 Further, let $k$ be the index such that $u_k$ is $\pleq$-maximal the among vertices of $U$, and let $C$ be the set of all vertices in the $k$th column of $G$. The choice of $k$ implies that $u\pleq c$ for all $u\in U$ and $c\in C$.
 
 We observe that in $G$ there exists a family ${\cal P}$ of $d$ pairwise vertex-disjoint paths, each connecting a vertex in $C$ with a vertex in $U$. Indeed, otherwise, by Menger's theorem, there would be a vertex subset $X$ with $|X|<d$ that would intersect every such path. But then there would exist a row and a column of $G$ that would not intersect $X$, while the union of this row and this column would contain a path connecting a vertex of $C$ with a vertex of $U$; a contradiction.
 Further, since the paths from ${\cal P}$ are pairwise vertex-disjoint, there are $d$ of them, and the whole graph $G$ contains $d^2$ vertices, we conclude that at most $d/2$ paths from ${\cal P}$ may contain more than $2d$ vertices. Therefore, we can find a subfamily ${\cal P'}\subseteq {\cal P}$ of size at least $d/2$ such that each $P\in {\cal P'}$ has length at most $2d-1$.
 
 For each path $P\in {\cal P'}$, let $w(P)$ be the first (i.e. closest to the endpoint in $C$) vertex of $P$ satisfying $w(P)\pleq c$ for all $c\in C$. Note that $w(P)$ is well-defined, since each $P\in {\cal P'}$ contains at least one vertex $w$ satisfying $w\pleq c$ for all $c\in C$: the endpoint of $P$ belonging to $U$ is such a vertex. Since paths of ${\cal P'}$ are pairwise vertex-disjoint, vertices $w(P)$ for $P\in {\cal P'}$ are pairwise different. Moreover, by concatenating a  prefix of $P$ from the endpoint on $C$ to $w(P)$ with a subpath of $C$ from $u_k$ to the said endpoint, we obtain a strong reachability path of length at most $3d-2$ that witness that $w(P)\in \SReach_{3d-2}[u_k]$. It follows that
 $$\scol_{3d-2}(G,\pleq)\geq |\SReach_{3d-2}[u_k]|\geq |{\cal P}'| \geq d/2.\qedhere$$
\end{proof}

\end{document}